\documentclass[11pt,a4paper]{article}

\usepackage{epsf,epsfig,amsfonts,amsgen,amsmath,amstext,amsbsy,amsopn,amsthm,cases,listings,color
}
\usepackage{ebezier,eepic}
\usepackage{color}
\usepackage{multirow}
\usepackage{epstopdf}
\usepackage{graphicx}   
\usepackage{pgf,tikz}
\usepackage{mathrsfs}
\usepackage[marginal]{footmisc}
\usepackage{enumerate}
\usepackage{enumitem}
\usepackage[titletoc]{appendix}
\usepackage{booktabs}
\usepackage{url}
\usepackage{mathtools}

\usepackage{pgfplots}
\usepackage{authblk}
\usepackage{amssymb}

\usepackage{wasysym}

\usepackage{empheq}

\usepackage{dsfont}

\usepackage{tikz}
\usepackage{longtable}
\usepackage{float}
\usepackage{subfigure}
\pgfplotsset{compat=1.18}
\usepackage{mathrsfs}
\usepackage{wasysym} 
\usetikzlibrary{arrows}
\usepackage{bm}
\usepackage{bbm}
\usepackage[T1]{fontenc}
\usepackage[backref=page]{hyperref}
\usepackage{stmaryrd}

\usepackage{multicol}

\usepackage[thinc]{esdiff}

\allowdisplaybreaks[1]

\definecolor{uuuuuu}{rgb}{0.27,0.27,0.27}
\definecolor{sqsqsq}{rgb}{0.1255,0.1255,0.1255}

\newtheorem{definition}{Definition} [section]
\newtheorem{theorem}[definition]{Theorem}
\newtheorem{lemma}[definition]{Lemma}

\newtheorem{corollary}[definition]{Corollary}

\newtheorem{claim}[definition]{Claim}
\newtheorem{problem}[definition]{Problem}

\makeatletter
\newsavebox\myboxA
\newsavebox\myboxB
\newlength\mylenA
\newcommand*\xoverline[2][0.75]{%
    \sbox{\myboxA}{$\m@th#2$}%
    \setbox\myboxB\null
    \ht\myboxB=\ht\myboxA%
    \dp\myboxB=\dp\myboxA%
    \wd\myboxB=#1\wd\myboxA
    \sbox\myboxB{$\m@th\overline{\copy\myboxB}$}
    \setlength\mylenA{\the\wd\myboxA}
    \addtolength\mylenA{-\the\wd\myboxB}%
    \ifdim\wd\myboxB<\wd\myboxA%
       \rlap{\hskip 0.5\mylenA\usebox\myboxB}{\usebox\myboxA}%
    \else
        \hskip -0.5\mylenA\rlap{\usebox\myboxA}{\hskip 0.5\mylenA\usebox\myboxB}%
    \fi}
\makeatother
    
\begin{document}
\title{\bf\Large Odd hypergraph Mantel theorems}
\author{
Jianfeng Hou \quad
Xizhi Liu \quad 
Yixiao Zhang \quad 
Hongbin Zhao \quad 
Tianming Zhu
\vspace{-1.5em} 
}
\date{\today}
%
\maketitle
\begin{abstract}
    A classical result of Sidorenko (1989) shows that the Tur\'{a}n density of every $r$-uniform hypergraph with three edges is bounded from above by $1/2$. For even $r$, this bound is tight, as demonstrated by Mantel's theorem on triangles and Frankl's theorem on expanded triangles. In this note, we prove that for odd $r$, the bound $1/2$ is never attained, thereby answering a question of Keevash and revealing a fundamental difference between hypergraphs of odd and even uniformity. Moreover, our result implies that the expanded triangles form the unique class of three-edge hypergraphs whose Tur\'{a}n density attains $1/2$. 
\end{abstract}

\renewcommand\thefootnote{} 
\addtocounter{footnote}{-1} 
\renewcommand\thefootnote{\arabic{footnote}} 

\section{Introduction}\label{SEC:Introduction}
Given an integer $r\ge 2$, an \emph{$r$-uniform hypergraph} (henceforth an \emph{$r$-graph}) $\mathcal{H}$ is a collection of $r$-subsets of some set $V$. 
We call $V$ the \emph{vertex set} of $\mathcal{H}$ and denote it by $V(\mathcal{H})$. 
When $V$ is understood, we usually identify a hypergraph $\mathcal{H}$ with its edge set and use $|\mathcal{H}|$ to denote the number of edges of $\mathcal{H}$. 

Given a family $\mathcal{F}$ of $r$-graphs, an $r$-graph $\mathcal{H}$ is \emph{$\mathcal{F}$-free} if it does not contain any member of $\mathcal{F}$ as a subgraph. 
The \emph{Tur\'{a}n number} $\mathrm{ex}(n,\mathcal{F})$ of $\mathcal{F}$ is the maximum number of edges in an $\mathcal{F}$-free $r$-graph on $n$ vertices, and the limit $\pi(\mathcal{F}) \coloneqq \lim_{n \to \infty} \mathrm{ex}(n,\mathcal{F})/\binom{n}{r}$ is called the \emph{Tur\'{a}n density} of $\mathcal{F}$. 
The study of the Tur\'{a}n numbers and Tur\'{a}n densities is a central topic in Extremal Combinatorics. 
Much is known when $r = 2$, thanks to the seminal works of Tur\'{a}n~\cite{Tur41} and Erd\H{o}s--Stone~\cite{ES46} (see also~\cite{ES66}). 
For $r \ge 3$, determining $\pi(\mathcal{F})$ is notoriously difficult in general, despite significant effort devoted to this area.
For results up to~2011, we refer the reader to the excellent survey by Keevash~\cite{Kee11}. 

In this note, we study the Tur\'{a}n density of $r$-graphs with exactly three edges. 
For the case $r=2$, the classical Mantel theorem~\cite{Mantel07} from 1907 shows that the Tur\'{a}n density of the triangle is $1/2$, and the tightness of this bound is witnessed by the balanced complete bipartite graphs. 
Define 
\begin{align*}
    \gamma(r)
    \coloneqq \max\left\{\pi(F) \colon \text{$F$ is an $r$-graph with exactly three edges} \right\}. 
\end{align*}
As an extension of Mantel's theorem to hypergraphs, a general theorem of Sidorenko~\cite{Sid89} shows that the Tur\'{a}n density of an $r$-graph with exactly three edges is always bounded from above by $1/2$, that is, $\gamma(r) \le 1/2$. 
For an integer $k \ge 1$, let 
\begin{align*}
    \mathbb{T}_{2k} 
    \coloneqq \left\{S_1 \cup S_2,~S_2\cup S_3,~S_3 \cup S_1 \right\} 
\end{align*}
denote the $2k$-uniform hypergraph (known as the \emph{expanded triangle}) on vertex set $S_1 \cup S_2 \cup S_3$, where $S_1, S_2, S_3$ are pairwise disjoint $k$-sets.  
A classical result of Frankl~\cite{Fra90} shows that $\pi(\mathbb{T}_{2k}) = 1/2$. 
In particular, $\gamma(r) = 1/2$ for even $r$, that is, Sidorenko's bound is tight when $r$ is even. 

In a work strengthening Sidorenko's result, Keevash~\cite{Kee05} posed the question of whether the bound $\gamma(r) \le 1/2$ is ever tight when $r$ is odd. 
This question was later raised again in~\cite{Kee11} and~\cite{MPS11}. 
Keevash also observed that for $r=3$, classical results of Bollob\'{a}s~\cite{Bo74} and de Caen~\cite{Cae83} show that $\gamma(3)$ is strictly smaller than $1/2$. 

In this note, we answer Keevash's question by showing that for odd $r$, the bound $1/2$ is never attained, thereby revealing an interesting difference between hypergraphs of even and odd uniformity.

\begin{theorem}\label{THM:odd-Mantel}
    Suppose that $r$ is a positive odd integer. Then $\gamma(r) < 1/2$. 
\end{theorem}

Let us remark that determining the exact value of $\gamma(r)$ for odd $r$ remains a very challenging problem. 
Indeed, for $r=3$, as observed by Keevash, results of Bollob\'{a}s~\cite{Bo74} and Frankl--F\"{u}redi~\cite{FF84} imply that $\gamma(3) = \pi(K_{4}^{3-})$, where $K_{4}^{3-}$ denotes the $3$-graph on four vertices with three edges. 
However, determining $\pi(K_{4}^{3-})$ remains an open problem, with the current best upper bound $\pi(K_4^{3-}) \le 0.2871\ldots$ established by Baber--Talbot~\cite{BT11} using the computer-assisted Flag Algebra machinery of Razborov~\cite{Raz07}.

Theorem~\ref{THM:odd-Mantel}, together with some additional arguments, yields the following corollary, which shows that the expanded triangles are the ``worst'' three-edge hypergraphs. 
\begin{corollary}\label{CORO:even-Mantel}
    Let $r$ be a positive even integer.
    Suppose that $F$ is an $r$-graph with exactly three edges such that $F \not\cong \mathbb{T}_{r}$.
    Then $\pi(F) < 1/2$.
\end{corollary}

This paper is organized as follows. 
Section~\ref{SEC:prelim} introduces the necessary definitions and preliminary results. 
In Section~\ref{SEC:Proof}, we present the proof of Theorem~\ref{THM:odd-Mantel}.
Finally, Section~\ref{SEC:Remarks} contains some concluding remarks. 

\section{Preliminaries}\label{SEC:prelim}
In this section, we present some definitions and preliminary results. 

Let $r \ge 2$ be an integer and $\mathcal{H}$ be an $r$-graph. 
For every vertex $v \in V(\mathcal{H})$, the \emph{link} of $v$ in $\mathcal{H}$ is defined as 
\begin{align*}
    L_{\mathcal{H}}(v)
    \coloneq \left\{S \in \tbinom{V(\mathcal{H}) \setminus \{v\}}{r-1} \colon S \cup \{v\} \in \mathcal{H} \right\}.
\end{align*}
%
The \emph{degree} of $v$ is given by $d_{\mathcal{H}}(v) \coloneqq |L_{\mathcal{H}}(v)|$.  
We use $\delta(\mathcal{H})$, $\Delta(\mathcal{H})$, and $d(\mathcal{H})$ to denote the \emph{minimum}, \emph{maximum}, and \emph{average} degrees of $\mathcal{H}$, respectively. 
The subscript $\mathcal{H}$ will be omitted when it is clear from the context.

Let $r \ge s \ge 2$ be integers. 
For an $s$-graph $F$, the \emph{$r$-suspension} of $F$, denoted by $\hat{F}^{r}$, is the $r$-graph given by 
\begin{align*}
    \hat{F}^{r}
    \coloneqq \left\{ C \cup e \colon e \in F \right\},  
\end{align*}
where $C$ is an $(r-s)$-set disjoint from $V(F)$. 

It is clear from the definition that for integers $r_2 > r_1 \ge s$, the $r_2$-graph $\hat{G}^{r_2}$ can be viewed as the $(r_2-r_1)$-suspension of the $r_1$-graph $\hat{G}^{r_1}$. 

We have the following simple lemma on the Tur\'{a}n density of suspensions. 

\begin{lemma}\label{LEMMA:suspension-Turan-density}
    Let $r \ge s \ge 2$ be integers and $F$ be an $s$-graph. 
    Then $\pi(\hat{F}^{r}) \le \pi(F)$. 
    Consequently, $\pi(\hat{F}^{r_2}) \le \pi(\hat{F}^{r_1})$ for all integers $r_2 \ge r_1 \ge s$. 
\end{lemma}
\begin{proof}[Proof of Lemma~\ref{LEMMA:suspension-Turan-density}]
    It suffices to consider the case $r = s+1$. 
    Let $\mathcal{H}$ be an $\hat{F}^{s+1}$-free $(s+1)$-graph on $n$ vertices with exactly $\mathrm{ex}(n,\hat{F}^{s+1})$ edges. 
    It follows from the definition of suspension that for every vertex $v \in V(\mathcal{H})$, the link $L_{\mathcal{H}}(v)$ is $F$-free. 
    Therefore, $d_{\mathcal{H}}(v) \le \mathrm{ex}(n-1,F)$, and consequently, 
    \begin{align*}
        \frac{\mathrm{ex}(n,\hat{F}^{s+1})}{\binom{n}{s+1}}
        = \frac{|\mathcal{H}|}{\binom{n}{s+1}}
        = \frac{1}{s+1} \frac{\sum_{v\in V(\mathcal{H})} d_{\mathcal{H}}(v)}{\binom{n}{s+1}}
        \le \frac{1}{s+1} \frac{n \cdot \mathrm{ex}(n-1,F)}{\binom{n}{s+1}}
        = \frac{\mathrm{ex}(n-1,F)}{\binom{n-1}{s}}. 
    \end{align*}
    Letting $n \to \infty$ in the above, we obtain $\pi(\hat{F}^{s+1}) \le \pi(F)$. 
\end{proof}

Given two $r$-graphs $F_{1}$ and $F_{2}$, a map $\phi \colon V(F_{1}) \to V(F_{2})$ is a \emph{homomorphism} from $F_{1}$ to $F_{2}$ if $\phi(e) \in F_{2}$ for every $e \in F_{1}$. 
We will use the following standard result in Tur\'{a}n problems, which follows from a result of Erd\H{o}s--Simonovits~\cite{ES83} (see also~{\cite[Theorem~2.2]{Kee11}}). 

\begin{lemma}\label{LEMMA:blowup-lemma}
    Let $F_{1}$ and $F_{2}$ be two $r$-graphs.
    Suppose that there exists a homomorphism from $F_{1}$ to $F_{2}$. 
    Then $\pi(F_1) \leq \pi(F_{2})$. 
\end{lemma}

Recall the definition of $\mathbb{T}_{2k}$ from the previous section. 
Let $V_1 \cup V_2 = V$ be a partition of a vertex set $V$. 
Denote by $\mathbb{B}^{2k}[V_1, V_2]$ the collection of all $2k$-subsets of $V$ that intersect both $V_1$ and $V_2$ in an odd number of vertices. 
We call $\mathbb{B}^{2k}[V_1, V_2]$ the \emph{complete odd-bipartite} $2k$-graph with parts $V_1$ and $V_2$. 
Denote by $\mathbb{B}^{2k}_{n}$ the complete odd-bipartite $2k$-graph on $n$ vertices with the maximum number of edges. 

Confirming a conjecture of Frankl~\cite{Fra90}, Keevash--Sudakov~\cite{KS05} proved that for large $n$, the unique extremal construction for $\mathbb{T}_{2k}$ is $\mathbb{B}_{n}^{2k}$. Moreover, they established the following stability theorem.  

\begin{theorem}[{\cite[Theorem 3.4]{KS05}}]\label{THM:T2k-stability}
    Let $k \ge 1$ be an integer. 
    For every $\varepsilon >0$, there exist $\delta_{\ref{THM:T2k-stability}} = \delta_{\ref{THM:T2k-stability}}(k, \varepsilon) >0$ and $N_{\ref{THM:T2k-stability}} = N_{\ref{THM:T2k-stability}}(k, \varepsilon)$ such that the following holds for every $n \ge N_{\ref{THM:T2k-stability}}$. 
    Suppose that $\mathcal{H}$ is a $\mathbb{T}_{2k}$-free $2k$-graph on $n$ vertices with $|\mathcal{H}| \ge \left(\tfrac{1}{2} - \delta_{\ref{THM:T2k-stability}}\right)\binom{n}{2k}$. 
    Then there exists a partition $V(\mathcal{H}) = V_1 \cup V_2$ with $\left| |V_1| - |V_2| \right| \le 1$ such that 
    \begin{align*}
        \left| \mathcal{H} \triangle \mathbb{B}^{2k}[V_1, V_2] \right| \le \varepsilon n^{2k}. 
    \end{align*}
\end{theorem}

Here, we use $A \triangle B$ to denote the symmetric difference of two sets $A$ and $B$. 

\section{Proof of Theorem~\ref{THM:odd-Mantel}}\label{SEC:Proof}
In this section, we present the proof of Theorem~\ref{THM:odd-Mantel}. 

\subsection{Reduction}\label{SUBSEC:Reduction}
In this subsection, we show that Theorem~\ref{THM:odd-Mantel} reduces to the following result. 

\begin{theorem}\label{THM:suspension-expanded-triangle}
    For every integer $k \ge 1$, we have $\pi(\hat{\mathbb{T}}_{2k}^{2k+1}) < 1/2$. 
\end{theorem}

For every integer $r \ge 2$, denote by $\mathcal{T}^{r}$ the collection of all $r$-graphs with exactly three edges. 
By removing isolated vertices if necessary, we may assume that $\delta(F) \ge 1$ for every $F \in \mathcal{T}^{r}$. 
We further partition $\mathcal{T}^{r}$ into the following two subfamilies: 
\begin{align*}
    \mathcal{T}_{1}^{r}
    \coloneqq \left\{F \in \mathcal{T}^{r} \colon \delta(F) = 1 \right\}
    \quad\text{and}\quad 
    \mathcal{T}_{2}^{r}
    \coloneqq \left\{F \in \mathcal{T}^{r} \colon \delta(F) \ge 2 \right\}. 
\end{align*}

\begin{lemma}\label{LEMMA:describe-T2}
    The following statements hold. 
    \begin{enumerate}[label=(\roman*)]
        \item\label{LEMMA:describe-T2-a} We have $\mathcal{T}_{2}^{r} = \big\{\hat{\mathbb{T}}_{2i}^{r} \colon 1 \le i \le \lfloor r/2 \rfloor \big\}$. 
        \item\label{LEMMA:describe-T2-b} For every $r$-graph $F_{1} \in \mathcal{T}_{1}^{r}$, there exists an $r$-graph $F_{2} \in \mathcal{T}_{2}^{r}$ such that there exists a homomorphism from $F_1$ to $F_2$.
    \end{enumerate}
\end{lemma}
\begin{proof}[Proof of Lemma~\ref{LEMMA:describe-T2}]
    First, we present the proof of Lemma~\ref{LEMMA:describe-T2}~\ref{LEMMA:describe-T2-a}. 
    We proceed by induction on~$r$. 
    The base case $r \in \{2,3\}$ holds, since it is easy to see that the family $\mathcal{T}_2^{2}$ consists only of the single graph $K_3 = \hat{\mathbb{T}}_{2}^{2}$, and the family $\mathcal{T}_2^{3}$ consists only of the single $3$-graph $K_4^{3-} = \hat{\mathbb{T}}_{2}^{3}$. 
    
    Now assume that $r \ge 4$. 
    First, we consider the case where $r = 2s$ is even. 
    Fix an arbitrary $2s$-graph $F = \left\{E_1, E_2, E_3\right\} \in \mathcal{T}_2^{2s}$. 
    
    If every vertex of $F$ has degree exactly two, then it is easy to see that $|V(F)| = 3s$, and hence, 
    \begin{align*}
        |E_1 \cap E_2| 
        = 3s - \big( |E_2 \cap E_3| + |E_1 \cap E_3| \big) 
        = 3s - |E_3| 
        = s.  
    \end{align*}
    By symmetry, we also have $|E_{1} \cap E_{3}| = |E_{2} \cap E_{3}| = s$, which implies that $F \cong \mathbb{T}_{2s}$. 
    
    If $F$ contains a vertex $x$ of degree three, then $F$ is the $2s$-suspension of the $(2s-1)$-graph $F' \coloneqq F - x$. Since $F' \in \mathcal{T}_2^{2s-1}$, it follows from the induction hypothesis that $F' \cong \hat{\mathbb{T}}_{2i}^{2s-1}$ for some $1 \le i \le s-1$. 
    Therefore, $F \cong \hat{\mathbb{T}}_{2i}^{2s}$, as desired. 

    Now we consider the case where $r = 2s+1$ is odd.
    Fix an arbitrary $(2s+1)$-graph $F = \left\{E_1, E_2, E_3\right\} \in \mathcal{T}_2^{2s+1}$. 
    A simple double-counting argument shows that the sum of the degrees of the vertices in $F$ is $3(2s+1) = 6s+3$, which is odd. 
    Therefore, there exist a vertex $x$ of degree three in $F$. 
    Hence, $F$ is the $(2s+1)$-suspension of the $2s$-graph $F' \coloneqq F - x$.
    Since $F' \in \mathcal{T}_2^{2s}$, it follows from the induction hypothesis that $F' \cong \hat{\mathbb{T}}_{2i}^{2s}$ for some $1 \le i \le s$. 
    Therefore, $F \cong \hat{\mathbb{T}}_{2i}^{2s+1}$, as desired.

    Next, we present the proof of Lemma~\ref{LEMMA:describe-T2}~\ref{LEMMA:describe-T2-b}. 
    Fix an $r$-graph $F_{1} =\{E_1, E_2, E_3\} \in \mathcal{T}_{1}^{r}$. 
    By the definition of $\mathcal{T}_{1}^{r}$, there exists a vertex $x$ of degree one in $F_1$. 
    By relabeling the edges of $F_1$ if necessary, we may assume that $x \in E_1$. 
    Since $E_2 \neq E_1$, there exists a vertex $y \in E_2 \setminus E_1$. 
    Let $\tilde{F}$ denote the $r$-graph obtained from $F_1$ by replacing the edge $E_1$ with $( E_1\setminus\{x\} ) \cup \{y\}$ (it is possible that $( E_1\setminus\{x\} ) \cup \{y\}$ is already contained in $F_1$, in which case we keep only one copy). 
    Note that $\tilde{F}$ has exactly two edges or the number of vertices of degree one in $\tilde{F}$ is strictly smaller than that of $F_1$, and the map $\phi \colon V(F_1) \to V(\tilde{F})$ given by 
    \begin{align}\label{EQ:map_phi}
        \phi(v) 
        = 
        \begin{cases}
            v, &\quad\text{if}\quad v \neq x, \\
            y, &\quad\text{if}\quad  v = x, 
        \end{cases}
    \end{align}
    is a homomorphism from $F_1$ to $\tilde{F}$. 
    By repeating this operation, we eventually obtain a homomorphism from $F_1$ to an $r$-graph $F$ in which every vertex has degree at least two, or to a single edge. 
    In the former case, we have $F \in \mathcal{T}_{2}^{r}$, and we are done. 
    In the latter case, there exists a homomorphism from $F_1$ to every member in $\mathcal{T}_{2}^{r}$, and we are also done. 
\end{proof}

Combining Lemma~\ref{LEMMA:blowup-lemma} with Lemma~\ref{LEMMA:describe-T2}~\ref{LEMMA:describe-T2-a} and~\ref{LEMMA:describe-T2-b}, we obtain for every integer $r \ge 2$, 
\begin{align}\label{equ:reduction-a}
    \max\left\{\pi(F) \colon F \in \mathcal{T}_{1}^{r} \right\}
    \le \max\left\{\pi(F) \colon F \in \mathcal{T}_{2}^{r} \right\}
    = \max\left\{ \pi(\hat{\mathbb{T}}_{2i}^{r}) \colon 1 \le i \le \lfloor \tfrac{r}{2} \rfloor \right\}. 
\end{align}
Combining this with Lemma~\ref{LEMMA:suspension-Turan-density}, we conclude that for every odd integer $r \ge 3$, 
\begin{align*}
    \max\left\{\pi(F) \colon F \in \mathcal{T}^{r} \right\}
    = \max\left\{ \pi(\hat{\mathbb{T}}_{2i}^{r}) \colon 1 \le i \le \lfloor \tfrac{r}{2} \rfloor \right\}
    \le \max\left\{ \pi(\hat{\mathbb{T}}_{2i}^{2i+1}) \colon 1 \le i \le \lfloor \tfrac{r}{2} \rfloor \right\}. 
\end{align*}

Therefore, the proof of Theorem~\ref{THM:odd-Mantel} reduces to that of Theorem~\ref{THM:suspension-expanded-triangle}. 

We end this subsection by presenting the proof of Corollary~\ref{CORO:even-Mantel}, assuming Theorem~\ref{THM:odd-Mantel}.
We will use the following strengthening of Lemma~\ref{LEMMA:describe-T2}~\ref{LEMMA:describe-T2-b}.

\begin{lemma}\label{LEMMA:describe-T2-strengthen}
    Let $r \ge 3$ be an integer. 
    For every $r$-graph $F_{1} \in \mathcal{T}_{1}^{r}$, there exists an $r$-graph $F_{2} \in \mathcal{T}_{2}^{r}$ with $\Delta(F_{2}) = 3$ such that there exists a homomorphism from $F_1$ to $F_2$.
\end{lemma}
\begin{proof}[Proof of Lemma~\ref{LEMMA:describe-T2-strengthen}]
    Since $r \ge 3$, by Lemma~\ref{LEMMA:describe-T2}~\ref{LEMMA:describe-T2-a}, there exists an $r$-graph $F \in \mathcal{T}_{2}^{r}$ with $\Delta(F) = 3$. 
    Fix an $r$-graph $F_{1} = \{E_1, E_2, E_3\} \in \mathcal{T}_{1}^{r}$. 
    If $\Delta(F_1) = 1$, then $F_1$ consists of three pairwise disjoint edges, which means that there exists a homomorphism from $F_1$ to every $r$-graph $F_2 \in \mathcal{T}_{2}^{r}$ with $\Delta(F_2) = 3$, and we are done. 
    If $\Delta(F_1) = 3$, then we are also done by Lemma~\ref{LEMMA:describe-T2}~\ref{LEMMA:describe-T2-b}. 
    Therefore, we may assume that $\Delta(F_1) = 2$.

    \begin{claim}\label{CLAIM:describe-T2-strengthen}
        There exist two vertices $x, y \in V(F_1)$ that are not contained in any edge of $F_1$, such that $d_{F_1}(x) = 1$ and $d_{F_1}(y) = 2$. 
    \end{claim}    
    \begin{proof}[Proof of Claim~\ref{CLAIM:describe-T2-strengthen}]
        Let $u \in V(F_1)$ be a vertex of degree one (which exists by the definition of $\mathcal{T}_{1}^{r}$), and let $v \in V(F_1)$ be a vertex of degree two (which exists since $\Delta(F_1) = 2$).  
        By symmetry, we may assume that $u \in E_1$.  
        Note that we are done if $v \not\in E_1$. 
        Thus, we may assume that $v \in E_1$ as well.  
        By symmetry, we may assume that $v \in E_1 \cap E_2$.  
        Since $\Delta(F_1) = 2$, we have $v \notin E_3$.  
    
        Suppose that $E_2 \cap E_3 \neq \emptyset$.  
        Choose a vertex $y \in E_2 \cap E_3$.  
        Since $\Delta(F_1) = 2$, the vertex $y$ is not contained in $E_1$.  
        Thus, the pair $\{u, y\}$ satisfies the claim.  
    
        Suppose instead that $E_2 \cap E_3 = \emptyset$.  
        Then there exists a vertex $w \in E_3$ of degree one in $F_1$.  
        The pair $\{w, v\}$ satisfies the claim. 
    \end{proof}

    Fix two vertices $x, y \in V(F_1)$ that are not contained in any edge of $F_1$, such that $d_{F_1}(x) = 1$ and $d_{F_1}(y) = 2$.
    By symmetry, we may assume that $x \in E_1$ and $y \in E_2 \cap E_3$.
    As in the proof of Lemma~\ref{LEMMA:describe-T2}~\ref{LEMMA:describe-T2-b}, let $\tilde{F}$ denote the $r$-graph obtained from $F_1$ by replacing the edge $E_1$ with  $( E_1\setminus\{x\} ) \cup \{y\}$ (it is possible that $( E_1\setminus\{x\} ) \cup \{y\}$ is already contained in $F_1$, in which case we keep only one copy).
    The map $\phi$ defined in~\eqref{EQ:map_phi} is a homomorphism from $F_1$ to $\tilde{F}$.
    
    Note that either $\tilde{F}$ has exactly two edges, or $\tilde{F}$ contains a vertex of degree three.
    In the former case, there exists a homomorphism from $F_1$ to a single edge, and consequently to every $r$-graph $F_2 \in \mathcal{T}_{2}^{r}$ with $\Delta(F_2) = 3$, and we are done. 
    In the latter case, $\tilde{F}$ contains a vertex of degree three, and we are also done by Lemma~\ref{LEMMA:describe-T2}~\ref{LEMMA:describe-T2-b}.
\end{proof}

We are now ready to present the proof of Corollary~\ref{CORO:even-Mantel}. 

\begin{proof}[Proof of Corollary~\ref{CORO:even-Mantel}]
    Fix a positive even integer $r$.
    Let $F$ be an $r$-graph with exactly three edges such that $F \not\cong \mathbb{T}_{r}$.
    If $r = 2$, then $F$ is bipartite (in fact, a forest), which implies that $\pi(F) = 0$.
    Therefore, we may assume that $r \ge 4$.

    Suppose that $F \in \mathcal{T}_{1}^{r}$.  
    Then by Lemma~\ref{LEMMA:describe-T2-strengthen}, there exists a homomorphism from $F$ to some $r$-graph $\tilde{F} \in \mathcal{T}^{r}_{2}$ containing a vertex $x$ of degree three.  
    It follows from Lemma~\ref{LEMMA:blowup-lemma} that $\pi(F) \le \pi(\tilde{F})$.  
    Since $\tilde{F}$ is the $r$-suspension of the $(r-1)$-graph $\tilde{F} - x$, it follows from Lemma~\ref{LEMMA:suspension-Turan-density} and Theorem~\ref{THM:odd-Mantel} (noting that $r-1$ is odd) that  
    \begin{align*}
        \pi(\tilde{F}) \le \pi(\tilde{F} - x) < \frac{1}{2},
    \end{align*}
    which implies that $\pi(F) < 1/2$.  
    
    Suppose that $F \in \mathcal{T}_{2}^{r}$.  
    Since $F \not\cong \mathbb{T}_{r}$, it follows from Lemma~\ref{LEMMA:describe-T2}~\ref{LEMMA:describe-T2-a} that $F \cong \hat{\mathbb{T}}_{2i}^{r}$ for some $i < r/2$.  
    Then by Lemma~\ref{LEMMA:suspension-Turan-density} and Theorem~\ref{THM:suspension-expanded-triangle}, we have  
    \begin{align*}
        \pi(F)
        = \pi(\hat{\mathbb{T}}_{2i}^{r})
        \le \pi(\hat{\mathbb{T}}_{2i}^{2i+1})
        < \frac{1}{2}. 
    \end{align*}
    This completes the proof of Corollary~\ref{CORO:even-Mantel}.   
\end{proof}

\subsection{Proof of Theorem~\ref{THM:suspension-expanded-triangle}}\label{SUBSEC:density-suspension-expanded-triangle}
In this subsection, we present the proof of Theorem~\ref{THM:suspension-expanded-triangle}.
For ease of notation, we set $\mathbb{T} \coloneqq \mathbb{T}_{2k}$ and $\hat{\mathbb{T}} \coloneqq \hat{\mathbb{T}}_{2k}^{2k+1}$, and we will also omit the superscript $2k$ from $\mathbb{B}^{2k}[V_1, V_2]$. 

\begin{proof}[Proof of Theorem~\ref{THM:suspension-expanded-triangle}]
    Fix an integer $k \ge 1$. 
    Let $\varepsilon > 0$ be a sufficiently small constant, and let $\delta \coloneqq \delta_{\ref{THM:T2k-stability}}(k, \varepsilon) > 0$ and $N \coloneqq N_{\ref{THM:T2k-stability}}(k, \varepsilon) > 0$ be the constants given by Theorem~\ref{THM:T2k-stability}. 
    By reducing $\delta$ if necessary, we may assume that $\delta \le \varepsilon$.
    We will prove that $\pi(\hat{\mathbb{T}}) \le \tfrac{1}{2} - \tfrac{\delta}{3}$. 
    
    Assume for contradiction that $\pi(\hat{\mathbb{T}}) > \tfrac{1}{2} - \tfrac{\delta}{3}$. 
    Let $n \gg N$ be a sufficiently large integer, and let $\mathcal{H}$ be a $\hat{\mathbb{T}}$-free $(2k+1)$-graph on $n$ vertices with exactly $\mathrm{ex}(n,\hat{\mathbb{T}})$ edges, that is, $\mathcal{H}$ is an extremal $(2k+1)$-graph. 
    We may choose $n$ sufficiently large so that
    \begin{align*}
        \left|\mathcal{H}\right| 
        = \mathrm{ex}(n,\hat{\mathbb{T}})
        \ge \left(\frac{1}{2} - \frac{\delta}{2}\right)\binom{n}{2k+1}. 
    \end{align*}
    Observe that every pair of vertices in $\hat{\mathbb{T}}$ is contained in some edge. 
    Thus, a standard deleting-duplicating argument (see e.g.~{\cite[Proposition 2.6]{BLLP25}}) shows that 
    \begin{align*}
        \Delta(\mathcal{H}) - \delta(\mathcal{H})
        \le \binom{n-2}{2k-1}. 
    \end{align*}
    Consequently, 
    \begin{align}\label{equ:mindeg-H}
        \delta(\mathcal{H})
        \ge d(\mathcal{H}) - \binom{n-2}{2k-1}
        = \frac{(2k+1) \cdot |\mathcal{H}|}{n} - \binom{n-2}{2k-1}
        \ge \left(\frac{1}{2} - \delta\right)\binom{n}{2k}. 
    \end{align}
    Let $x \in V(\mathcal{H})$ be a vertex. 
    Since $\mathcal{H}$ is $\hat{\mathbb{T}}$-free, it follows from the definition of $\hat{\mathbb{T}}$ that the link $L(x)$ is $\mathbb{T}$-free. 
    Combining it with~\eqref{equ:mindeg-H} and Theorem~\ref{THM:T2k-stability}, we conclude that there exists a partition $V_x^1 \cup V_x^2 = V \coloneqq V(\mathcal{H})$ with $\left||V_x^1|-|V_x^2|\right| \le 1$ such that 
    \begin{align}\label{equ:T2k-stability-link}
        \left|L(x) \triangle  \mathbb{B}[V_x^1,V_x^2] \right| \le \varepsilon n^{2k}. 
    \end{align}
    Define the set of bad edges and missing edges (of $L(x)$) as 
    \begin{align*}
        \mathcal{B}_{x}
        \coloneqq L(x) \setminus \mathbb{B}[V_x^1,V_x^2] 
        \quad\text{and}\quad 
        \mathcal{M}_{x}
        \coloneqq \mathbb{B}[V_x^1,V_x^2] \setminus L(x). 
    \end{align*}
    It follows from~\eqref{equ:T2k-stability-link} that 
    \begin{align}\label{equ:upper-bound-B-M}
        \max \{|\mathcal{B}_{x}|,~|\mathcal{M}_{x}|\}
        \le |\mathcal{B}_{x}| + |\mathcal{M}_{x}| 
        = \left|L(x) \triangle \mathbb{B}[V_x^1,V_x^2] \right| 
        \le \varepsilon n^{2k}. 
    \end{align} 
    Define 
    \begin{align*}
        Z_x \coloneqq 
        \left\{v \in V(\mathcal{H}) \colon d_{\mathcal{M}_{x}}(v) \ge  \varepsilon^{1/2} n^{2k-1}\right\}. 
    \end{align*}
    
    \begin{claim}\label{CLAIM:few-mindegree-vtx}
        We have $\left|Z_{x}\right| \le 2k \varepsilon^{1/2}n$. 
    \end{claim}
    \begin{proof}[Proof of Claim~\ref{CLAIM:few-mindegree-vtx}]
        By definition of $Z_{x}$, we have 
        \begin{align*}
            |\mathcal{M}_{x}| 
            = \frac{1}{2k} \sum_{v \in V} d_{\mathcal{M}_{x}}(v)
            \ge \frac{1}{2k} |Z_{x}| \cdot \varepsilon^{1/2} n^{2k-1}. 
        \end{align*}
        Combining this with~\eqref{equ:upper-bound-B-M}, we obtain $|Z_{x}|\le 2k \varepsilon^{1/2}n$. 
    \end{proof}

    For convenience, let $\mathcal{P}_x \coloneqq \{V_{x}^{1}, V_{x}^{2}\}$ denote the bipartition corresponding to the vertex $x$.
    For a pair of vertices $x,y \in V$, define 
    \begin{align*}
        \mathrm{dist}(\mathcal{P}_{x}, \mathcal{P}_{y})
        \coloneq \min \left\{|V_x^1 \triangle V_y^1|,~|V_x^1 \triangle V_y^2|\right\}. 
    \end{align*}
    
    \begin{claim}\label{CLAIM:small-diff}
        For every pair of vertices $x, y \in V$ with $y \in V \setminus Z_{x}$, we have 
        \begin{align*}
            \mathrm{dist}(\mathcal{P}_{x}, \mathcal{P}_{y}) 
            \le \varepsilon^{1/8k} n. 
        \end{align*}
    \end{claim}
    \begin{proof}[Proof of Claim~\ref{CLAIM:small-diff}]
        For each $(i,j) \in [2] \times [2]$, define $U_{ij} \coloneqq V_x^i \cap V_y^j$. 
        Suppose for contradiction that Claim~\ref{CLAIM:small-diff} fails. 
        Then
        \begin{align*}
            |U_{12}| + |U_{21}| 
            = |V_x^1 \triangle V_y^1| 
            \ge \mathrm{dist}(\mathcal{P}_{x}, \mathcal{P}_{y}) 
            > \varepsilon^{1/8k} n. 
        \end{align*}
        Since both partitions $V_{x}^{1} \cup V_{x}^{2} = V$ and $V_{y}^{1} \cup V_{y}^{2} = V$ are balanced, we have 
        \begin{align*}
            \bigl||U_{12}| - |U_{21}|\bigr|
            = \bigl|\big(|V_x^1| - |U_{11}|\big) - \big(|V_y^1|-|U_{11}| \big) \bigr|
            = \bigl| |V_x^1| - |V_y^1| \bigr|
            \le 1. 
        \end{align*}
        Consequently, 
        \begin{align*}
            \min\left\{ |U_{12}|,~|U_{21}| \right\} 
            \ge \frac{1}{2} |V_x^1 \triangle V_y^1| - 1
            \ge \frac{\varepsilon^{1/8k} n}{3}.
        \end{align*}
        For the same reason, the inequality also holds for $U_{11}$ and $U_{22}$. In conclusion, for every $(i,j) \in [2] \times [2]$, 
        \begin{align}\label{equ:Uij-lower-bound}
            |U_{ij}| 
            \ge \frac{\varepsilon^{1/8k} n}{3} 
            > |Z_{x}|. 
        \end{align}

        Assume by symmetry that $y \in U_{11}$. 
        Choose uniformly at random a vertex $z$ and a $(k-1)$-set $S_1$ from $U_{21}$ with $z \not\in S_1$, and independently, choose uniformly at random a $(k-1)$-set $S_2$ and a $k$-set $S_3$ from $U_{22}$ with $S_2 \cap S_3 = \emptyset$.
        Let 
        \begin{itemize}
            \item $A_1$ be the event that $\{y\}\cup S_1 \cup S_3 \in L(x)$, 
            \item $A_2$ be the event that $\{y,z\}\cup S_1 \cup S_2 \in L(x)$, and 
            \item $A_3$ be the event that $\{z\} \cup S_2 \cup S_3 \in L(y)$. 
        \end{itemize}

        Since $y \in V \setminus Z_{x}$, we have from the definition of $Z_{x}$ that $d_{\mathcal{M}_{x}}(y) < \varepsilon^{1/2} n^{2k-1}$. 
        Combining this with~\eqref{equ:Uij-lower-bound}, we conclude that the probability that $A_1$ does not occur satisfies 
        \begin{align*}
            \mathbb{P}\left[ \xoverline{A_{1}} \right]
            \le \frac{\varepsilon^{1/2}n^{2k-1}}{\binom{|U_{21}|-1}{k-1}\binom{|U_{22}|-k+1}{k}} 
            & \le \frac{(k-1)!k!\varepsilon^{1/2}n^{2k-1}}{\left(\tfrac{\varepsilon^{1/8k}n}{3}-2k+2\right)^{2k-1}} \\
            & \le \frac{(k-1)!k!\varepsilon^{1/2}n^{2k-1}}{\left({\varepsilon^{1/8k}n}/{4}\right)^{2k-1}}
            \le 4^{2k-1}(k!)^2 \varepsilon^{\frac{2k+1}{8k}}
            < \frac{1}{4}. 
        \end{align*}
        Similarly, the probability that $A_2$ does not occur satisfies 
        \begin{align*}
            \mathbb{P}\left[ \xoverline{A_{2}} \right]
            \le \frac{\varepsilon^{1/2}n^{2k-1}}{\left(|U_{21}|-k+1\right)\binom{|U_{21}|-1}{k-1}\binom{|U_{22}|-k}{k-1}} 
            & \le \frac{\left((k-1)!\right)^2\varepsilon^{1/2}n^{2k-1}}{\left(\tfrac{\varepsilon^{1/8k}n}{3}-2k+2\right)^{2k-1}} \\
            & \le 4^{2k-1}(k!)^2 \varepsilon^{\frac{2k+1}{8k}}
            < \frac{1}{4}. 
        \end{align*}
        It follows from~\eqref{equ:upper-bound-B-M} and~\eqref{equ:Uij-lower-bound} that the probability that $A_3$ does not occur satisfies
        \begin{align*}
            \mathbb{P}\left[ \xoverline{A_{3}} \right]
            \le \frac{\varepsilon n^{2k}}{\left(|U_{21}|-k+1\right)\binom{|U_{22}|-k}{k-1}\binom{|U_{22}|-k+1}{k}} 
            \le \frac{(k-1)!k!\varepsilon n^{2k}}{\left( \tfrac{\varepsilon^{1/8k}n}{3} -2k+2\right)^{2k}} 
            \le 4^{2k}(k!)^2\varepsilon^{\frac{3}{4}}
            < \frac{1}{4}. 
        \end{align*}
        Therefore, the probability that $A_1$, $A_2$, and $A_3$ occur simultaneously satisfies
        \begin{align*}
            \mathbb{P}\left[ A_1\cap A_2 \cap A_3 \right] 
            \ge 1- \mathbb{P}\left[ \xoverline{A_{1}} \right] - \mathbb{P}\left[ \xoverline{A_{2}} \right] - \mathbb{P}\left[ \xoverline{A_{3}} \right] 
            > \frac{1}{4}. 
        \end{align*}
        Fix a selection of $(z, S_1, S_2, S_3)$ such that $A_1$, $A_2$, and $A_3$ occur simultaneously.  
        Let 
        \begin{align*} 
            E_1 \coloneqq \{x, y\}\cup S_1 \cup S_3, \quad 
            E_2 \coloneqq \{x, y,z\}\cup S_1 \cup S_2, \quad\text{and}\quad 
            E_3 \coloneqq \{y, z\} \cup S_2 \cup S_3. 
        \end{align*}
        Note that $\{E_1, E_2, E_3\} \subseteq \mathcal{H}$ forms a copy of $\hat{\mathbb{T}}$, with $y$ serving as the vertex of degree three, which is a contradiction.
    \end{proof}
    
    Fix a vertex $x \in V$, and let $W_1 \coloneqq V_x^1\setminus Z_{x}$ and $W_2 \coloneqq V_x^2\setminus Z_{x}$. 
    It follows from Claim~\ref{CLAIM:few-mindegree-vtx} that 
    \begin{align*}
        \min\{|W_1|,~|W_2|\} 
        \ge \min\{|V_{x}^{1}| - |Z_{x}|,~|V_{x}^{2}| - |Z_{x}|\}
        \ge \lfloor n/2 \rfloor - 2k \varepsilon^{1/2} n. 
    \end{align*}
    For every vertex $y \in V \setminus Z_{x}$, it follows from Claim~\ref{CLAIM:small-diff} that 
    \begin{align*}
        \mathrm{dist}(\mathcal{P}_x, \mathcal{P}_y)
        = \min \left\{|V_x^1 \triangle V_y^1|,~|V_x^1 \triangle V_y^2|\right\} 
        \le \varepsilon^{1/8k} n. 
    \end{align*}
    By relabeling the two parts $V_y^1$ and $V_y^2$ if necessary, we may assume that it is always the case that
    \begin{align}\label{equ:Vx1-Vz1-diff}
        |V_x^1 \triangle V_y^1|
        = \mathrm{dist}(\mathcal{P}_x, \mathcal{P}_y)
        \le \varepsilon^{1/8k} n. 
    \end{align}

    \begin{claim}\label{CLAIM:bad-link-small}
        For every vertex $y \in V \setminus Z_{x}$, we have 
        \begin{align*}
            \left| L(y) \triangle \mathbb{B}[W_1, W_2] \right| 
            \le 5k \varepsilon^{1/8k} n^{2k}. 
        \end{align*}
    \end{claim}
    \begin{proof}[Proof of Claim~\ref{CLAIM:bad-link-small}]
        It follows from~\eqref{equ:T2k-stability-link},~\eqref{equ:Vx1-Vz1-diff}, and Claim~\ref{CLAIM:small-diff} that 
        \begin{align*}
            \left| L(y) \triangle \mathbb{B}[W_1, W_2] \right| 
            & \le \left| L(y) \triangle \mathbb{B}[V_{y}^{1}, V_{y}^{2}] \right| + \left| \mathbb{B}[V_{y}^{1}, V_{y}^{2}] \triangle \mathbb{B}[W_1, W_2] \right| \\
            & \le \varepsilon n^{2k} + \big(|W_1 \triangle V_{y}^{1}| + |W_2 \triangle V_{y}^{2}| \big) \tbinom{n}{2k-1} \\
            & \le \varepsilon n^{2k} + \big( |W_1 \triangle V_{x}^{1}| + |V_{x}^{1} \triangle V_{y}^{1}| +  |W_2 \triangle V_{x}^{2}| + |V_{x}^{2} \triangle V_{y}^{2}| \big) \tbinom{n}{2k-1} \\
            & = \varepsilon n^{2k} + \left( 2 \cdot  \mathrm{dist}(\mathcal{P}_{x}, \mathcal{P}_{y}) + |Z_{x}|\right) \tbinom{n}{2k-1} \\
            & \le \varepsilon n^{2k} + \left( 2 \varepsilon^{1/8k}n + 2k \varepsilon^{1/2} n \right) \tbinom{n}{2k-1}
            \le 5k \varepsilon^{1/8k} n^{2k}, 
        \end{align*}
        which proves Claim~\ref{CLAIM:bad-link-small}. 
    \end{proof}

    Next, we show that there exists a vertex $y \in W_1$ with $\left| L(y) \triangle \mathbb{B}[W_1, W_2] \right| > 5k \varepsilon^{1/8k} n^{2k}$. 
    This contradicts Claim~\ref{CLAIM:bad-link-small}, and thus completing the proof of Theorem~\ref{THM:suspension-expanded-triangle}. 

    \begin{claim}\label{CLAIM:bad-link-large}
        There exists a vertex $y \in W_1 \subseteq V \setminus Z_{x}$ such that
        \begin{align*}
            \left| L(y) \triangle \mathbb{B}[W_1, W_2] \right| 
            \ge \frac{n^{2k}}{4^{2k} (2k)!} 
            > 5k \varepsilon^{1/8k} n^{2k}. 
        \end{align*}
    \end{claim}
    \begin{proof}[Proof of Claim~\ref{CLAIM:bad-link-large}]
        Define an auxiliary bipartite graph $G$ on $W_1 \cup W_2$ where a pair $(y,z) \in W_{1} \times W_{2}$ forms an edge in $G$ iff $y \in V_{z}^{1} \setminus Z_{z}$. 
        It follows from Claim~\ref{CLAIM:few-mindegree-vtx} and~\eqref{equ:Vx1-Vz1-diff} that for every $z \in W_2$, we have 
        \begin{align*}
            d_G(z)
             = \left| \big(V_{z}^{1} \setminus Z_{z}\big) \cap W_1 \right| 
            & \ge |V_{x}^{1}| - |Z_{x}| - |V_x^1 \triangle V_z^1| - |Z_{z}|  \\
            & \ge \lfloor n/2 \rfloor - 4k \varepsilon^{1/2}n - \varepsilon^{1/8k}n
            \ge \frac{n}{2} - 5k\varepsilon^{1/8k}n. 
        \end{align*}
        
        It follows that there exists a vertex $y \in W_1$ with 
        \begin{align}\label{equ:dG-y-lower-bound}
            d_{G}(y)
            \ge \frac{|G|}{|W_1|}
            & \ge \frac{|W_2| \left(n/2 - 5k\varepsilon^{1/8k}n \right)}{|W_1|} \notag \\
            & \ge \frac{\left( \lfloor n/2 \rfloor - 2k\varepsilon^{1/2}n \right)\left(n/2 - 5k\varepsilon^{1/8k}n \right) }{\lceil n/2 \rceil}
            > \frac{n}{4}. 
        \end{align}
        Let $\hat{W} \coloneqq N_{G}(y) \subseteq W_{2}$. 
        For each vertex $w \in \hat{W}$, consider the link of the pair $\{y,w\}$ in $\mathcal{H}$, restricted to the set $W_{2}$, defined by
        \begin{align*}
            \mathcal{L}_{w}
            \coloneqq \left\{ e \in \tbinom{W_2}{2k-1} \colon e \cup \{y\} \in L(w) \right\}
            = \left\{ e \in \tbinom{W_2}{2k-1} \colon e \cup \{w, y\} \in \mathcal{H} \right\}. 
        \end{align*}
        Observe that for each $e \in \mathcal{L}_{w}$, we have $e \cup \{w\} \in \tbinom{W_2}{2k} \cap L(y)$, and hence, 
        \begin{align}\label{equ:Ly-on-W2-lower-bound}
            \left| \tbinom{W_2}{2k} \cap L(y) \right|
            \ge \frac{1}{2k} \sum_{w\in \hat{W}} |\mathcal{L}_{w}|. 
        \end{align}
        Fix a vertex $w \in \hat{W}$. 
        Recall from the definition of $G$ that $y \in V_{w}^{1} \setminus Z_{w}$. 
        It follows from the definition of complete odd-bipartite hypergraph that every $(2k-1)$-set in $W_2 \cap V_{w}^{2}$ forms an edge with $y$ in $\mathbb{B}[V_{w}^{1}, V_{w}^{2}]$. 
        Combining this with the definition of $Z_{w}$, we obtain 
        \begin{align}\label{equ:Lw-lower-bound}
            |\mathcal{L}_{w}|
            \ge \binom{|W_2 \cap V_{w}^{2}|}{2k-1} - d_{\mathcal{M}_{w}}(y) 
            & \ge \binom{|W_2| - \mathrm{dist}(\mathcal{P}_{x}, \mathcal{P}_{w})}{2k-1} - d_{\mathcal{M}_{w}}(y) \notag \\
            & \ge \binom{\lfloor n/2 \rfloor - 2k\varepsilon^{1/2} n - \varepsilon^{1/8k} n}{2k-1} - \varepsilon^{1/2} n^{2k-1}   \notag \\
            & \ge \binom{n/3}{2k-1} - \varepsilon^{1/2} n^{2k-1} 
            \ge \frac{1}{(2k-1)!} \left(\frac{n}{4}\right)^{2k-1}. 
        \end{align}
        Notice from the definition of complete odd-bipartite hypergraph that 
        \begin{align*}
            \tbinom{W_2}{2k} \cap L(y) 
            \subseteq L(y) \triangle \mathbb{B}[W_1, W_2]. 
        \end{align*}
        So it follows from~\eqref{equ:dG-y-lower-bound},~\eqref{equ:Ly-on-W2-lower-bound}, and~\eqref{equ:Lw-lower-bound} that 
        \begin{align*}
            \left| L(y) \triangle \mathbb{B}[W_1, W_2] \right|
            \ge \left| \tbinom{W_2}{2k} \cap L(y) \right| 
            & \ge \frac{1}{2k} \sum_{w\in \hat{W}} |\mathcal{L}_{w}| \\
            & \ge \frac{|\hat{W}|}{2k} \frac{1}{(2k-1)!} \left(\frac{n}{4}\right)^{2k-1}  \\
            & \ge \frac{n}{8k} \frac{1}{(2k-1)!} \left(\frac{n}{4}\right)^{2k-1} 
            = \frac{n^{2k}}{4^{2k} (2k)!},   
        \end{align*}
        which proves Claim~\ref{CLAIM:bad-link-large}. 
    \end{proof}
    We now see that Claim~\ref{CLAIM:bad-link-small} and Claim~\ref{CLAIM:bad-link-large} contradict each other, and thus this completes the proof of Theorem~\ref{THM:suspension-expanded-triangle}.
\end{proof}

\section{Concluding remarks}\label{SEC:Remarks}
Using Levente's recently developed flag algebra package (see~\cite{BLLP24,BLLP25,BP25indu,BP25cube,BDHLZ25} for details and applications), one can easily show\footnote[2]{The calculations and generated certificates can be found at \href{https://github.com/xliu2022/xliu2022.github.io/tree/main/FlagAlgebra_Certificate}{\url{https://github.com/xliu2022/xliu2022.github.io/tree/main/FlagAlgebra_Certificate}}.} that 
\begin{align*}
    \pi(\hat{\mathbb{T}}_{2}^{5}) \le \frac{1}{5}
    \quad\text{and}\quad 
    \pi(\hat{\mathbb{T}}_{4}^{5}) \le \frac{152}{499}. 
\end{align*}
So it follows from~\eqref{equ:reduction-a} that for every $5$-graph $F$ with three edges, 
\begin{align*}
    \pi(F)
    \le \max\left\{\pi(\hat{\mathbb{T}}_{2}^{5}),~\pi(\hat{\mathbb{T}}_{4}^{5})\right\}
    \le \frac{152}{499}
    = 0.304609\ldots, 
\end{align*}
which is notably below $1/2$. 

Using the entropy approach introduced recently by Chao--Yu~\cite{CY24}, together with further arguments, we can show the following explicit upper bound for $\pi(\hat{\mathbb{T}}_{2i}^{r})$, and hence for $\gamma(r)$.

\begin{theorem}\label{THM:triangle-entropy}
    For positive integers $r$ and $i$ satisfying $r \ge 2i$, we have 
    \begin{align*}
        \pi(\hat{\mathbb{T}}_{2i}^{r})
        \le \frac{i}{r}. 
    \end{align*}
    Consequently, $\gamma(r) \le {\lfloor r/2 \rfloor}/{r}$. 
\end{theorem}

Recall from Lemma~\ref{LEMMA:suspension-Turan-density} that $\pi(\hat{F}^{s+1}) \le \pi(F)$ for every $s$-graph $F$.
This motivates the following natural questions. 

\begin{problem}\label{PROB:suspension-Turan-density}
    Let $s \ge 2$ be an integer.
    \begin{enumerate}[label=(\roman*)]
        \item\label{PROB:suspension-Turan-density-a} Characterize all $s$-graphs $F$ for which $\pi(\hat{F}^{s+1}) < \pi(F)$.
        \item\label{PROB:suspension-Turan-density-b} Is it true that, for every $s$-graph $F$, there exists a constant $r$ such that $\pi(\hat{F}^{r}) < \pi(F)$?
        \item\label{PROB:suspension-Turan-density-c} Determine, for every $s$-graph $F$, the limit $\lim_{r\to \infty} \pi(\hat{F}^{r})$. 
    \end{enumerate}
\end{problem}

The proof of Theorem~\ref{THM:suspension-expanded-triangle} seems to provide a useful approach to addressing this problem.
In forthcoming work, we will show that Problem~\ref{PROB:suspension-Turan-density}~\ref{PROB:suspension-Turan-density-a} holds for every non-bipartite graph $F$ in the case $s=2$.

Inequality~\eqref{equ:reduction-a} motivates the following problem, for which results on hypergraph Mantel theorems such as~\cite{FF89,GS17,Sid24,CY24,Liu25,IY25,MZ25} may be useful.

\begin{problem}\label{PROB:worst-Ti}
    Let $r \ge 3$ be an odd integer. 
    Determine the value of $i_{\ast} \le \lfloor r/2 \rfloor$ for which 
    \begin{align*}
        \pi(\hat{\mathbb{T}}_{2i_{\ast}}^{r})
        = \max\left\{ \pi(\hat{\mathbb{T}}_{2i}^{r}) \colon 1 \le i \le \lfloor r/2 \rfloor \right\}.
    \end{align*}
\end{problem}
\section*{Acknowledgments}
J.H. was supported by the National Key R\&D Program of China (No.~2023YFA1010202) 
and by the Central Guidance on Local Science and Technology Development Fund of Fujian Province (No.~2023L3003). 
X.L. was supported by the Excellent Young Talents Program (Overseas) of the National Natural Science Foundation of China. 

\bibliographystyle{alpha}
\bibliography{Oddtriangle}


\section*{\normalfont Authors}
\begin{multicols}{2}
\begin{flushleft}
\vbox{%
Jianfeng Hou \\
{\small Center for Discrete Mathematics} \\
{\small Fuzhou University} \\
{\small Fuzhou, 350108, China} \\
\texttt{jfhou@fzu.edu.cn}
}

\vspace{0.7cm}

\vbox{%
Xizhi Liu \\
{\small School of Mathematical Sciences} \\
{\small University of Science and Technology of China} \\
{\small Hefei, 230026, China} \\
\texttt{liuxizhi@ustc.edu.cn}
}

\vspace{0.7cm}

\vbox{%
Yixiao Zhang \\
{\small Center for Discrete Mathematics} \\
{\small Fuzhou University} \\
{\small Fuzhou, 350108, China} \\
\texttt{fzuzyx@gmail.com}
}

\vspace{0.7cm}

\vbox{%
Hongbin Zhao \\
{\small Center for Discrete Mathematics} \\
{\small Fuzhou University} \\
{\small Fuzhou, 350108, China} \\
\texttt{hbzhao2024@163.com}
}

\vspace{0.7cm}

\vbox{%
Tianming Zhu \\
{\small School of Mathematical Sciences} \\
{\small University of Science and Technology of China} \\
{\small Hefei, 230026, China} \\
\texttt{zhutianming@mail.ustc.edu.cn}
}
\end{flushleft}
\end{multicols}
\end{document}